\documentclass{amsart}
\usepackage{amsmath}
\usepackage{amsthm}
\usepackage{amssymb}

\usepackage{diagrams}
\input xy
\xyoption{all}

 \newtheorem{thm}{Theorem}[section]

 \newtheorem{prop}[thm]{Proposition}
 \theoremstyle{definition}
 
 \theoremstyle{remark}
 \newtheorem{rem}[thm]{Remark}
 \numberwithin{equation}{section}

\newcommand{\ZZ}{\mathbb{Z}}
\newcommand{\QQ}{\mathbb{Q}}

\newcommand{\LL}{\mathbb{L}}
\newcommand{\CC}{\mathbb{C}}

\newcommand{\Aa}{\mathbb{A}}

\newcommand{\Hom}{\mathrm{Hom}}

\newcommand{\Ext}{\mathrm{Ext}}
\newcommand{\Biext}{\mathrm{Biext}}
\newcommand{\uHom}{\underline{\mathrm{Hom}}}

\newcommand{\bBiext}{\mathrm{\mathbf{Biext}}}
\newcommand{\bR}{\mathrm{\mathbf{R}}}

\newcommand{\W}{\mathrm{W}}

\newcommand{\T}{\mathrm{T}}

\newcommand{\cD}{\mathcal{D}}
\newcommand{\cC}{\mathcal{C}}

\newcommand{\cH}{\mathcal{H}}
\newcommand{\cO}{\mathcal{O}}

\newcommand{\DMeffgm}{\mathrm{DM}^{\mathrm{eff}}_{\mathrm{gm}}}

\newcommand{\DMeffmeno}{\mathrm{DM}^{\mathrm{eff}}_{-}}
\newcommand{\DMeffmenoet}{\mathrm{DM}^{\mathrm{eff}}_{-,{\mathrm{\acute{e}t}}}}

\newcommand{\cMR}{\mathcal{MR}}

\newcommand{\Sm}{\mathrm{Sm}}
\newcommand{\SmCor}{\mathrm{SmCor}}
\newcommand{\Sh}{\mathrm{Sh}_{\mathrm{Nis}}}

\begin{document}

\title[Biextensions and 1-motives]
{Biextensions of 1-motives in Voevodsky's category of motives}

\author{Cristiana Bertolin and Carlo Mazza}

\address{NWF-I Mathematik, Universit\"at Regensburg, D-93040 Regensburg}
\email{cristiana.bertolin@mathematik.uni-regensburg.de}

\address{Dip. di Matematica, Universit\`a di Genova, Via Dodecaneso 35, I-16133 Genova}
\email{mazza@dima.unige.it}



%\date{}
%\dedicatory{}

%\commby{}

%%% ----------------------------------------------------------------------

\begin{abstract}
Let $k$ be a perfect field. In this paper we prove that biextensions of 1-motives define multilinear morphisms between 1-motives in Voevodsky's triangulated category ${\DMeffgm} (k,{\QQ})$ of effective geometrical motives over $k$ with rational coefficients.
\end{abstract}

%%% ----------------------------------------------------------------------

\maketitle

%%% ----------------------------------------------------------------------

\tableofcontents

\section*{Introduction}

Let $k$ be a perfect field. In~\cite{O} Orgogozo constructs a fully faithful functor
\begin{equation}\label{orgogozo}
 {\cO}: {\cD}^b(1-\mathrm{Isomot}(k)) \longrightarrow {\DMeffgm} (k,{\QQ})
\end{equation}
from the bounded derived category of the category $1-\mathrm{Isomot}(k)$ of 1-motives over $k$ defined modulo isogenies to Voevodsky's triangulated category ${\DMeffgm} (k,{\QQ})$ of effective geometrical motives over $k$ with rational coefficients. If $M_i$ (for $i=1,2,3$) is a 1-motive defined over $k$ modulo isogenies, in this paper we prove that the group of isomorphism classes of biextensions of $(M_1,M_2)$ by $M_3$ is isomorphic to the group of morphisms of the category ${\DMeffgm} (k, {\QQ})$ from the tensor product ${\cO}(M_1)\otimes_{tr} {\cO}(M_2)$ to ${\cO}(M_3)$:

\begin{thm}\label{mainthm}
 Let $M_i$ (for $i=1,2,3$) be a 1-motive defined over a perfect field $k$. Then
\[ {\Biext}^1(M_1,M_2;M_3) \otimes {\QQ} \cong {\Hom}_{{\DMeffgm} (k,{\QQ})}({\cO}(M_1) \otimes_{tr} {\cO}(M_2), {\cO}(M_3)) . \]
\end{thm}

This isomorphism answers a question raised by Barbieri-Viale and Kahn in~\cite{BK1} Remark 7.1.3 2). In loc. cit. Proposition 7.1.2 e) they prove the above theorem in the case where $M_3$ is a semi-abelian variety. Our proof is a generalization of theirs.

If $k$ is a field of characteristic 0 embeddable in $\CC,$ by~\cite{D} (10.1.3) we have a fully faithful functor
\begin{equation}\label{realfunct}
    {\T}: \mathrm{1-Mot}(k)  \longrightarrow  {\cMR}(k)
\end{equation}
from the category $\mathrm{1-Mot}(k)$ of 1-motives over $k$ to the Tannakian category ${\cMR}(k)$ of mixed realizations over $k$ (see~\cite{J} I 2.1), which attaches to each 1-motive its Hodge realization for any embedding $k \hookrightarrow {\CC}$, its de Rham realization, its $\mathbf{\ell}$-adic realizations for any prime number $\ell$, and its comparison isomorphisms.
According to~\cite{B1} Theorem 4.5.1, if
$M_i$ (for $i=1,2,3$) is a 1-motive defined over $k$ modulo isogenies, the group of isomorphism classes of biextensions of $(M_1,M_2)$ by $M_3$ is isomorphic to the group of morphisms of the category ${\cMR}(k)$ from the tensor product ${\T}(M_1)\otimes {\T}(M_2)$ of the realizations of $M_1$ and $M_2$ to the realization ${\T}(M_3)$ of $M_3$. Putting together this result with Theorem~\ref{mainthm}, we get the following isomorphisms
\begin{eqnarray} \label{real-voe}
 \nonumber    {\Biext}^1(M_1,M_2;M_3)\otimes {\QQ}  &\cong& {\Hom}_{{\DMeffgm}(k;{\QQ})}({\cO}(M_1) \otimes_{tr} {\cO}(M_2), {\cO}(M_3))\\
 &\cong&
  {\Hom}_{{\cMR}(k)}\big({\T}(M_1)\otimes {\T}(M_2), {\T}(M_3)\big).
 \end{eqnarray}
These isomorphisms fit into the following context: in~\cite{H} Huber constructs a functor
 \[{\cH}: {\DMeffgm}(k,{\QQ})   \longrightarrow  {\cD}({\cMR}(k)) \]
 from Voevodsky's category ${\DMeffgm}(k,{\QQ})$ to the triangulated category  ${\cD}({\cMR}(k))$ of mixed realizations over $k$, which respects the tensor structures. Extending the functor ${\T}$ (\ref{realfunct}) to the derived category ${\cD}^b(1-\mathrm{Isomot}(k))$, we obtain the following diagram
\begin{equation}\label{dia}
\xymatrix{
  {\cD}^b(1-\mathrm{Isomot}(k))  \ar[r]^{\T} \ar[d]_{\cO} & {\cD}({\cMR}(k)) \\
   {\DMeffgm}(k, {\QQ}) \ar[ur]_{\cH} &
}
\end{equation}
The isomorphisms~(\ref{real-voe}) mean that biextensions of 1-motives define in a compatible way bilinear morphisms between 1-motives in each category involved in the above diagram.
Barbieri-Viale and Kahn informed the authors that in~\cite{BK2}
they have proved the commutativity of the diagram~(\ref{dia}) in an axiomatic setting. If $k={\CC}$, they can prove its commutativity without assuming axioms.

We finish generalizing Theorem~\ref{mainthm} to multilinear morphisms between 1-motives.

%--------------------------------------------------------
\section*{Acknowledgment}
The authors are very grateful to Barbieri-Viale and Kahn for several useful remarks improving the first draft of this paper.
 
%------------------------------------------------------
\section*{Notation}

If $C$ is an additive category, we denote by $C \otimes {\QQ}$ the associated $\QQ$-linear category which is universal for functors from $C$ to a $\QQ$-linear category. Explicitly, the category $C \otimes {\QQ}$ has the same objects as
the category $C$, but the sets of arrows of $C \otimes {\QQ}$ are the sets of arrows of $C$ tensored with $\QQ$, i.e. ${\Hom}_{C \otimes {\QQ}}(-,-)={\Hom}_{C}(-,-) \otimes_{\ZZ} {\QQ}$.

We give a quick review of Voevodsky's category of motives (see~\cite{V}).
 Denote by ${\Sm}(k)$ the category of smooth varieties over a field $k$. Let $A={\ZZ}$ or $\QQ$ be the coefficient ring. Let ${\SmCor} (k,A)$ be the category whose objects are smooth varieties over $k$ and whose morphisms are finite correspondences with coefficients in $A$. It is an additive category.

The \textbf{triangulated category ${\DMeffgm}(k,A)$ of effective geometrical motives over $k$} is the pseudo-abelian envelope of the localization of the homotopy category ${\cH}^b({\SmCor} (k,A))$ of bounded complexes over ${\SmCor} (k,A)$ with respect to the thick subcategory generated by the complexes $X \times_k {\Aa}^1_k \rightarrow X$ and $U \cap V \rightarrow U \oplus V \rightarrow X$ for any smooth variety $X$ and any Zariski-covering $X=U \cup V$.

The \textbf{category of Nisnevich sheaves on ${\Sm}(k)$}, ${\Sh}({\Sm}(k))$, is the category of abelian sheaves on ${\Sm}(k)$ for the Nisnevich topology.

A presheaf with transfers on ${\Sm}(k)$ is an additive contravariant functor from ${\SmCor}(k,A)$ to the category of abelian groups. It is called a Nisnevich sheaf with transfers if the corresponding presheaf of abelian groups on ${\Sm}(k)$ is a sheaf for the Nisnevich topology. Denote by ${\Sh}({\SmCor}(k,A))$ the \textbf{category of Nisnevich sheaves with transfers}. By~\cite{V} Theorem 3.1.4 it is an abelian category.

A presheaf with transfers $F$ is called homotopy invariant if for any smooth variety $X$ the natural map $F(X) \rightarrow F(X \times_k {\Aa}^1_k)$ induced by the projection $X \times_k {\Aa}^1_k \rightarrow X$ is an isomorphism.
A Nisnevich sheaf with transfers is called homotopy invariant if it is homotopy invariant as a presheaf with transferts.

The \textbf{category ${\DMeffmeno}(k,A)$ of effective motivic complexes} is the full subcategory of the derived category ${\cD}^-({\Sh}({\SmCor}(k,A)))$ of complexes of Nisnevich sheaves with transfers bounded from the above, which consists of complexes with homotopy invariant cohomology sheaves.

There exists a functor $ L: {\SmCor}(k,A) \rightarrow {\Sh}({\SmCor}(k,A)) $
which associates to each smooth variety $X$ a Nisnevich sheaf with transfers given by $L(X)(U)= c(U,X)_A$, where $c(U,X)_A$ is the free $A$-module generated by prime correspondences from $U$ to $X$. This functor extends to complexes furnishing a functor
\[ L: {\cH}^b({\SmCor} (k,A)) \longrightarrow {\cD}^-({\Sh}({\SmCor}(k,A))). \]

There exists also a functor $ C_*: {\Sh}({\SmCor}(k,A)) \rightarrow {\DMeffmeno}(k,A) $
which associates to each Nisnevich sheaf with transfers $F$ the effective motivic complex $C_* (F)$ given by $C_n(F)(U)= F(U \times \Delta^n)$ where $\Delta^*$ is the standard cosimplicial object. This functor extends to a functor
\begin{equation}\label{C}
    {\bR} C_*:  {\cD}^-({\Sh}({\SmCor}(k,A))) \longrightarrow {\DMeffmeno}(k,A)
\end{equation}
which is left adjoint to the natural embedding. Moreover, this functor identifies
the category ${\DMeffmeno}(k,A)$ with the localization of ${\cD}^-({\Sh}({\SmCor}(k,A)))$ with respect to the localizing subcategory generated by complexes of the form $L(X \times_k {\Aa}^1_k) \rightarrow L(X)$ for any smooth variety $X$ (see~\cite{V} Proposition 3.2.3).

If $X$ and $Y$ are two smooth varieties over $k$, the equality
\begin{equation}\label{tensor}
    L(X) \otimes L(Y) = L(X \times_k Y)
\end{equation}
defines a tensor structure on the category ${\Sh}({\SmCor}(k,A))$, which extends to the derived category ${\cD}^-({\Sh}({\SmCor}(k,A))).$
The tensor structure on ${\DMeffmeno}(k,A)$, that we denote by $\otimes_{tr}$, is the descent with respect to the projection $ {\bR} C_*$ (\ref{C}) of the tensor structure on ${\cD}^-({\Sh}({\SmCor}(k,A)))$ .

If we assume $k$ to be a perfect field, by~\cite{V} Proposition 3.2.6 there exists a functor
\begin{equation}\label{embedding}
    i: {\DMeffgm}(k,A)  \rightarrow {\DMeffmeno}(k,A)
\end{equation}
which is a full embedding with dense image and which makes the following diagram commutative
\[
\begin{array}{ccc}
{\cH}^b({\SmCor} (k,A))& \stackrel{L}{\rightarrow}& {\cD}^-({\Sh}({\SmCor}(k,A))) \\
 \downarrow &  &  \downarrow {\bR}C_* \\
{\DMeffgm}(k,A)   & \stackrel{i}{\dashrightarrow} & {\DMeffmeno}(k,A).
\end{array}
\]

\begin{rem} For Voevodsky's theory of motives with rational coefficients, the \'etale topology gives the same motivic answer as the Nisnevich topology: if we construct the category of effective motivic complexes using the \'etale topology instead of the Nisnevich topology, we get a triangulated category ${\DMeffmenoet}(k,A)$ which is equivalent as triangulated category to the category ${\DMeffmeno}(k,A)$ if we assume $A= {\QQ}$ (see~\cite{V} Proposition 3.3.2).
\end{rem}

%------------------------------------------------------
\section{1-motives in Voevodsky's category}

A \textbf{1-motive $M=(X,A,T,G,u)$} over a field $k$ (see~\cite{D} \S 10) consists of

\begin{itemize}
\item a group scheme $X$ over $k$, which is locally for the \'etale
topology, a constant group scheme defined by a finitely generated free
$\ZZ$-module,
\item an extention $G$ of an abelian $k$-variety $A$ by a $k$-torus $T,$
\item a morphism $u:X \longrightarrow G$ of commutative $k$-group schemes.
\end{itemize}

A 1-motive $M=(X,A,T,G,u)$ can be viewed also as a length 1 complex $[X \stackrel{u}{\rightarrow}G]$ of commutative $k$-group schemes. In this paper, as a complex we shall put $X$ in degree 0 and $G$ in degree 1. A morphism of 1-motives is a morphism of complexes of commutative $k$-group schemes. Denote by $1-\mathrm{Mot}(k)$ the category of 1-motives over $k$. It is an additive category but it isn't an abelian category.

Denote by $1-\mathrm{Isomot}(k)$ the $\QQ$-linear category $1-\mathrm{Mot}(k) \otimes {\QQ}$ associated to the category of 1-motives over $k$.
 The objects of $1-\mathrm{Isomot}(k)$ are called 1-isomotifs and the morphisms of $1-\mathrm{Mot}(k)$ which become isomorphisms in $1-\mathrm{Isomot}(k)$ are the
isogenies between 1-motives, i.e. the morphisms of complexes
$[X  \rightarrow G] \rightarrow [X' \rightarrow G']$ such that
 $X \rightarrow X'$ is injective with finite cokernel, and
$G \rightarrow G'$ is surjective with finite kernel.
The category $1-\mathrm{Isomot}(k)$ is an abelian category (see~\cite{O} Lemma 3.2.2).

 Assume now $k$ to be a perfect field. The two main ingredients which furnish the link between 1-motives and Voevodsky's motives are:
\begin{enumerate}
  \item any commutative $k$-group scheme represents a Nisnevich sheaf with transfers, i.e. an object of ${\Sh}({\SmCor}(k,A))$ (\cite{O} Lemma 3.1.2),
  \item if $A$ (resp. $T$, resp. $X$) is an abelian $k$-variety (resp. a $k$-torus, resp. a group scheme over $k$, which is locally for the \'etale topology, a constant group scheme defined by a finitely generated free $\ZZ$-module), then the Nisnevich sheaf with transfers that it represents is homotopy invariant (\cite{O} Lemma 3.3.1).
\end{enumerate}
Since we can view 1-motives as complexes of smooth varieties over $k$, we have a functor from the category of 1-motives to the category ${\cC}({\Sm}(k))$ of complexes over ${\Sm}(k)$. According to (1), this functor factorizes through the category
of complexes over ${\Sh}({\SmCor}(k,A))$:
\[ 1-\mathrm{Mot}(k) \longrightarrow {\cC}({\Sh}({\SmCor}(k,A)))\]
If we tensor with ${\QQ}$, we get an additive exact functor between abelian categories
\[ 1-\mathrm{Isomot}(k) \longrightarrow {\cC}({\Sh}({\SmCor}(k,A))\otimes {\QQ}).\]
Taking the associated bounded derived categories, we obtain a triangulated functor
 \[ {\cD}^b( 1-\mathrm{Isomot}(k)) \longrightarrow {\cD}^b({\Sh}({\SmCor}(k,A))\otimes {\QQ}).\]
Finally, according to (2) this last functor factorizes through the triangulated functor
 \[ {\cO}: {\cD}^b( 1-\mathrm{Isomot}(k)) \longrightarrow {\DMeffmeno}(k,A)\otimes {\QQ}.\]
 By~\cite{O} Proposition 3.3.3 this triangulated functor is fully faithful, and by loc. cit. Theorem 3.4.1 it factorizes through the thick subcategory $ d_1 {\DMeffgm}(k,{\QQ})$ of ${\DMeffgm}(k,{\QQ})$ generated by smooth varieties of dimension $\leq 1$ over $k$ and it induces an equivalence of triangulated categories, that we denote again by ${\cO}$,
 \[ {\cO}: {\cD}^b( 1-\mathrm{Isomot}(k)) \longrightarrow d_1 {\DMeffgm}(k,{\QQ}).\]
In order to simplify notation, if $M$ is a 1-motive, we denote again by $M$ its image in $d_1 {\DMeffgm}(k,{\QQ})$ through the above equivalence of categories and also its image in ${\DMeffmeno}(k,A)$ through the full embedding~(\ref{embedding}).

For the proof of Theorem~\ref{mainthm}, we will need the following

\begin{prop}\label{prop}
Let $M_i$ (for $i=1,2,3$) be a 1-motive defined over $k$.
The forgetful triangulated functors
\[ {\DMeffmeno}(k,A) \stackrel{a}{\hookrightarrow} {\cD}^-({\Sh}({\SmCor}(k,A)))  \stackrel{b}{\hookrightarrow} {\cD}^-({\Sh}({\Sm}(k)))\]
induce an isomorphism
\[ {\Hom}_{{\DMeffmeno}(k,A)}(M_1 \otimes_{tr} M_2,M_3) \cong {\Hom}_{{\cD}^-({\Sh}({\Sm}(k)))} (M_1 {\buildrel {\scriptscriptstyle \LL} \over \otimes} M_2,M_3). \]
\end{prop}

\begin{proof} The forgetful functor $a$ admits as left adjoint the functor ${\bR}C_*$ (\ref{C}). The forgetful functor $b$ from the category of Nisnevich sheaves with transfers to the category of Nisnevich sheaves admits as left adjoint the free sheaf with transfers functor
\begin{equation}\label{phi}
    \Phi:  {\cD}^-({\Sh}({\Sm}(k))) \longrightarrow  {\cD}^-({\Sh}({\SmCor}(k,A)))
\end{equation}
(\cite{V} Remark 1 page 202). If $X$ is a smooth variety over $k$,
let ${\ZZ}(X)$ be the sheafification with respect to the Nisnevisch topology of the presheaf $U \mapsto {\ZZ}[ {\Hom}_{{\Sm}(k)}(U,X)].$ Clearly
$\Phi({\ZZ}(X))$ is the Nisnevich sheaf with transfers $L(X)$.
If $Y$ is another smooth variety over $k$, we have that ${\ZZ}(X) \otimes {\ZZ}(Y) = {\ZZ}(X \times_k Y)$ (see~\cite{M} Lemma 12.14)
and so by formula~(\ref{tensor}) we get
\[ \Phi ({\ZZ}(X) \otimes {\ZZ}(Y))=\Phi({\ZZ}(X)) \otimes_{tr} \Phi({\ZZ}(Y)).\]
The tensor structure on ${\DMeffmeno}(k,A)$ is the descent of the tensor structure on
${\cD}^-({\Sh}({\SmCor}(k,A)))$ with respect to $ {\bR} C_*$ and therefore
\[ {\bR} C_* \circ \Phi ({\ZZ}(X) \otimes {\ZZ}(Y))= {\bR} C_* \circ \Phi({\ZZ}(X)) \otimes_{tr} {\bR} C_* \circ \Phi({\ZZ}(Y)).\]
Using this equality and the fact that the composite ${\bR} C_* \circ \Phi$ is the left adjoint of $b \circ a$, we have
\begin{eqnarray}
\nonumber {\Hom}_{{\cD}^-({\Sh}({\Sm}(k)))} (M_1 {\buildrel {\scriptscriptstyle \LL} \over \otimes} M_2,M_3) & \cong &
{\Hom}_{{\DMeffmeno}(k,A)}( {\bR} C_* \circ \Phi (M_1 {\buildrel {\scriptscriptstyle \LL} \over \otimes} M_2),M_3) \\
\nonumber   & \cong &  {\Hom}_{{\DMeffmeno}(k,A)}( {\bR} C_* \circ \Phi (M_1)  \otimes_{tr}  {\bR} C_* \circ \Phi (M_2),M_3).
\end{eqnarray}
Since 1-motives are complexes of homotopy invariant Nisnevich sheaves with transferts, the counit arrows ${\bR} C_* \circ \Phi (M_i) \rightarrow M_i $ (for $i=1,2$) are isomorphisms and so we can conclude.
\end{proof}

%------------------------------------------------------
\section{Bilinear morphisms between 1-motives}

Let $K_i=[A_i \stackrel{u_i}{\rightarrow} B_i]$ (for $i=1,2,3$) be a length 1 complex of abelian sheaves (over any topos $\mathrm{{\mathbf{T}}}$) with $A_i$ in degree 1 and $B_i$ in degree 0. A \textbf{biextension $({\mathcal{B}}, \Psi_1, \Psi_2,\lambda)$ of $(K_1,K_2)$ by $K_3$} consists of
\begin{enumerate}
    \item a biextension of $\mathcal{B}$ of $(B_1, B_2)$ by $B_3$;
    \item a trivialization $\Psi_1$ (resp. $\Psi_2$) of the biextension $(u_1,id_{B_2})^*{\mathcal{B}}$
    of $(A_1, B_2)$ by $B_3$ (resp. of the biextension $(id_{B_1},u_2)^*{\mathcal{B}}$ of $(B_1, A_2)$ by $B_3$) obtained as pull-back of ${\mathcal{B}}$ via $(u_1,id_{B_2}): A_1 \times B_2 \rightarrow B_1 \times B_2$ (resp. via $(id_{B_1}, u_2): B_1 \times A_2 \rightarrow B_1 \times B_2$ ). These two trivializations  have to coincide over $A_1 \times A_2$;
    \item a morphism $\lambda: A_1 \otimes A_2 \rightarrow A_3$ such that the composite $ A_1 \otimes A_2 \stackrel{\lambda}{\longrightarrow} A_3 \stackrel{u_3}{\longrightarrow} B_3$ is compatible with the restriction over $A_1 \times A_2$ of the trivializations $\Psi_1$ and $\Psi_2$.
\end{enumerate}
We denote by ${\bBiext}(K_1,K_2;K_3)$ the category of biextensions
of $(K_1,K_2)$ by $K_3$. The Baer sum of extensions defines a group law for the objects of the category ${\bBiext}(K_1,K_2;K_3)$, which is therefore a Picard category (see~\cite{SGA7} Expos\'e VII 2.4, 2.5 and 2.6).
Let ${\Biext}^0(K_1,K_2;K_3)$ be the group of automorphisms of any biextension of $(K_1,K_2)$ by $K_3$, and let ${\Biext}^1(K_1,K_2;K_3)$ be the group of isomorphism classes of biextensions of $(K_1,K_2)$ by $K_3$.

According to the main result of~\cite{B2}, we have the following homological interpretation of the groups ${\Biext}^i(K_1,K_2;K_3)$:
\begin{equation}\label{K}
    {\Biext}^i(K_1,K_2;K_3) \cong {\Ext}^i(K_1 {\buildrel {\scriptscriptstyle \LL} \over \otimes} K_2,K_3) \qquad (i=0,1)
\end{equation}

Since we can view 1-motives as complexes of commutative $S$-group schemes of length 1, all the above definitions apply to 1-motives.

\begin{rem}\label{homolo}
The homological interpretation (\ref{K}) of biextensions computed in~\cite{B2} is done for chain complexes $K_i=[A_i \stackrel{u_i}{\longrightarrow} B_i]$ with $A_i$ in degree 1 and $B_i$ in degree 0. In this paper 1-motives are considered as cochain complexes  $M_i=[X_i \stackrel{u_i}{\rightarrow}G_i]$ with $X$ in degree 0 and $G$ in degree 1. Therefore after switching from homological notation to cohomological notation, the homological interpretation of the group ${\Biext}^1(M_1,M_2;M_3)$ can be stated as follow:
\[
    {\Biext}^1(M_1,M_2;M_3) \cong {\Ext}^1(M_1[1] {\buildrel {\scriptscriptstyle \LL} \over \otimes} M_2[1],M_3[1])
\]
where the shift functor $[i]$ on a cochain complex $C^*$ acts as $(C^*[i])^j=C^{i+j}$.
\end{rem}

\par\noindent \emph{Proof of Theorem~\ref{mainthm}} By proposition~\ref{prop}, we have that
\begin{eqnarray}
\nonumber {\Hom}_{{\DMeffgm} (k,{\QQ})} (M_1 \otimes_{tr} M_2, M_3) & \cong &
{\Hom}_{{\DMeffmeno}(k,A)\otimes {\QQ}}( M_1 \otimes_{tr} M_2, M_3) \\
\nonumber   & \cong & {\Hom}_{{\cD}^-({\Sh}({\Sm}(k)))} (M_1 {\buildrel {\scriptscriptstyle \LL} \over \otimes} M_2,M_3) \otimes {\QQ}.
\end{eqnarray}
On the other hand, according to the remark~\ref{homolo} we have the following homological interpretation of the group ${\Biext}^1(M_1,M_2;M_3)$:
\[ {\Biext}^1(M_1,M_2;M_3) \cong {\Ext}^1(M_1[1] {\buildrel {\scriptscriptstyle \LL} \over \otimes} M_2[1],M_3[1]) \cong
    {\Hom}_{{\cD}^-({\Sh}({\Sm}(k)))} (M_1 {\buildrel {\scriptscriptstyle \LL} \over \otimes} M_2,M_3)\]
and so we can conclude.

%------------------------------------------------------
\section{Multilinear morphisms between 1-motives}

1-motives are endowed with an increasing filtration, called the weight filtration.
Explicitly, the weight filtration $W_*$ on a 1-motive $M=[X \stackrel{u}{\rightarrow} G] $ is
\begin{eqnarray}
\nonumber  {\W}_{i}(M) &=& M  ~~~~{\rm  for ~ each~} i \geq 0, \\
\nonumber  {\W}_{-1}(M) &=& [0 \longrightarrow G], \\
\nonumber  {\W}_{-2}(M) &=& [0 \longrightarrow  Y(1)], \\
\nonumber  {\W}_{j}(M) &=& 0 ~~~~{\rm  for ~ each~} j \leq -3.
\end{eqnarray}
Defining ${\rm Gr}_{i}^{{\W}}= {\W}_i / {\W}_{i+1}$, we have  ${\rm Gr}_{0}^{{\W}}(M)=
[X \rightarrow 0], {\rm Gr}_{-1}^{{\W}}(M)=
[0  \rightarrow A]$ and $ {\rm Gr}_{-2}^{{\W}}(M)=
[0  \rightarrow  Y(1)].$ Hence locally constant group schemes, abelian varieties and tori are the pure 1-motives underlying $M$ of weights 0,-1,-2 respectively.

The main property of morphisms of motives is that they have to respect the weight filtration, i.e. any morphism $f:A \rightarrow B$ of motives satisfies the following equality
\[ f(A) \cap {\W}_i(B) =  f({\W}_i(A)) \qquad \forall~~ i\in {\ZZ}. \]

Assume $M$ and $M_1, \dots, M_l$ to be 1-motives over a perfect field $k$ and consider a morphism
\[ F: \otimes^l_{j=1}M_j \rightarrow M.\]
Because morphisms of motives have to respect the weight filtration,
the only non trivial components of the morphism $F$ are the components of the morphism
\[ \otimes^l_{j=1}M_j \Big/ {\W}_{-3}(\otimes^l_{j=1}M_j) \longrightarrow M.\]
Using~\cite{B1} Lemma 3.1.3 with $i=-3$, we can write explicitly this last morphism in the following way
\[ \sum_{ \iota_1 < \iota_2 ~\mathrm{and}~ \nu_1< \dots < \nu_{l-2}
 \atop \iota_1 , \iota_2 \notin \{\nu_1, \dots,\nu_{l-2}\} }
 X_{\nu_1}\otimes \dots \otimes X_{\nu_{l-2}}\otimes ( M_{\iota_1} \otimes M_{\iota_2} /{\W}_{-3}( M_{\iota_1} \otimes M_{\iota_2})) \longrightarrow M.\]
To have the morphism
$$X_{\nu_1}\otimes \dots \otimes X_{\nu_{l-2}}\otimes ( M_{\iota_1} \otimes M_{\iota_2} /{\W}_{-3}( M_{\iota_1} \otimes M_{\iota_2})) \longrightarrow M$$
is equivalent to have the morphism
$$ M_{\iota_1} \otimes M_{\iota_2} /{\W}_{-3}( M_{\iota_1} \otimes M_{\iota_2}) \longrightarrow X_{\nu_1}^{\vee}\otimes \dots \otimes X_{\nu_{l-2}}^{\vee}\otimes M$$
where $X_{\nu_n}^{\vee}$ is the $k$-group scheme ${\uHom}( X_{\nu_n},{\ZZ})$ for $n=1, \ldots,l-2$. But as observed in~\cite{B1} \S 1.1 ``to tensor a motive by a motive of weight zero'' means to take a certain number of copies of this motive, and so applying Theorem~\ref{mainthm} we get

\begin{thm}
Let $M$ and $M_1, \dots, M_l$ be 1-motives over a perfect field $k$. Then,
\[ {\Hom}_{{\DMeffgm}(k,{\QQ})}(M_1 \otimes_{tr} M_2 \otimes_{tr} \cdots \otimes_{tr} M_l, M)  \cong \]
\[ \sum {\Biext}^1(M_{\iota_1},M_{\iota_2};X_{\nu_1}^{\vee}\otimes \dots \otimes X_{\nu_{l-2}}^{\vee} \otimes M)\otimes {\QQ} \]
where the sum is taken over all the $(l-2)$-uplets $\{\nu_1, \dots,\nu_{l-i+1}\}$ and all the 2-uplets $\{\iota_1,\iota_{2}\}$ of $\{1, \cdots,l\}$ such that $ \{\nu_1, \dots,\nu_{l-2}\}\cap \{\iota_1, \iota_{2}\} = \emptyset$ and
$ \nu_1 < \dots < \nu_{l-2}$, $  \iota_1 < \iota_{2}.$
\end{thm}

%---------------------------------------

%-----------------------------------------------------------

\end{document}